\documentclass[a4paper,11pt]{amsart}

\usepackage[utf8]{inputenc}
\usepackage{enumerate}
\usepackage{amsfonts,amsmath}
\usepackage{amssymb}
\usepackage{amsthm,amscd}
\usepackage{setspace} 
\usepackage{hyperref}
\usepackage{cite}
\usepackage{enumerate}

\newtheorem{thm}{Theorem}[section]
 \newtheorem{cor}[thm]{Corollary}
 \newtheorem{lem}[thm]{Lemma}
 \newtheorem{prop}[thm]{Proposition}
 \theoremstyle{definition}
 
 \theoremstyle{remark}
 \newtheorem{rem}[thm]{Remark}
 
\def\a{{\alpha}}
\def\b{{\beta}}
\def\g{{\gamma}}
\def\G{{\Gamma}}
\def\d{{\delta}}
\def\la{{\lambda}}
\def\La{{\Lambda}}
\def\f{{\phi}}
\def\M{{\mathfrak M}}
\def\Z{{\mathbb Z}}
\def\Q{{\mathbb Q}}
\def\R{{\mathbb R}}
\def\deg{{\rm deg}}
\def\ID2{$(\text{ID}_2)$}

\begin{document}

\title[]{Minimal Pr\"ufer-Dress rings and \\ products of idempotent matrices}

\author{Laura Cossu}

\address{Laura Cossu\\ Dipartimento di Matematica ``Tullio Levi-Civita''\\
 Via Trieste 63 - 35121 Padova, Italy}

\email{lcossu@math.unipd.it}

\author{Paolo Zanardo}

\address{Paolo Zanardo\\ Dipartimento di Matematica ``Tullio Levi-Civita''\\ Via Trieste 63 - 35121 Padova, Italy}

\email{pzanardo@math.unipd.it}

 

\thanks{Research supported by Dipartimento di Matematica ``Tullio Levi-Civita", Università di Padova under Progetto SID 2016 BIRD163492/16 “Categorical homological methods in the study of algebraic structures” and Progetto DOR1690814 “Anelli e categorie di moduli”. The first author is a member of the Gruppo Nazionale per le Strutture Algebriche, Geometriche e le loro Applicazioni (GNSAGA) of the Istituto Nazionale di Alta Matematica (INdAM)}

\begin{abstract}
We investigate a special class of Pr\"ufer domains, firstly introduced by Dress in 1965. 
The {\it minimal Dress ring} $D_K$, of a field $K$, is the smallest subring of $K$ that contains every element of the form $1/(1+x^2)$, with $x\in K$. We show that, for some choices of $K$, $D_K$ may be a valuation domain, or, more generally, a B\'ezout domain admitting a weak algorithm. Then we focus on the minimal Dress ring $D$ of $\R(X)$: we describe its elements, we prove that it is a Dedekind domain and we characterize its non-principal ideals. Moreover, we study the products of $2\times 2$ idempotent matrices over $D$, a subject of particular interest for Pr\"ufer non-B\'ezout domains.
\end{abstract}

\maketitle

\section*{Introduction}

It is well-known that the class $\mathcal{P}$ of Pr\"ufer domains is as important as large, and that several natural questions related to $\mathcal P$ are still open. (A most significant one is whether every B\'ezout domain is an elementary divisor ring; see \cite{FS} Ch. III.6.)

A nice subclass of $\mathcal P$ was discovered by Dress in the 1965 paper \cite{D}. 
Let $K$ be a field; a subring $D$ of $K$ is said to be a Pr\"ufer-Dress ring (or simply a Dress ring) if $D$ contains every element of the form ${1/(1 + x^2)}$ for  $x \in K$.

Of course, many important Pr\"ufer domains (the integers $\Z$, to name one) are not Dress rings. Nonetheless, the result in \cite{D} that any Dress ring is actually Pr\"ufer (see the first section) is quite useful. Indeed, it clearly furnishes a method for constructing relevant examples of Pr\"ufer domains (e.g., see \cite{Gilmer69}, \cite{Swan_n_gen}). 

For any assigned field $K$, we may consider the {\it minimal Dress ring} $D_K$ of $K$, namely
$
D_K = \Z_{\chi} [1/(1 + x^2) :   x \in K]
$
where $\Z_\chi$ is the prime subring of $K$.  
In this paper we mainly deal with minimal Dress rings $D_K$. In the first section we give some general results. We show that $D_K$ may coincide with $K$ (e.g., for $K = \R$), and that $D_K$ may be a valuation domain or a pull-back of a valuation domain. For $K = \R(\mathcal A)$, $\mathcal A$ a set of indeterminates, we also examine the relations between $D_K$ and the Pr\"ufer domains investigated by Sch\"ulting \cite{Schulting} (1979). Recall that Sch\"ulting was the first to prove the existence of Pr\"ufer domains $R$ containing finitely generated ideals that are neither principal nor two-generated (see also \cite{OR} for a neat discussion on this subject). As a by-product of the results in \cite{Schulting}, we get that also some minimal Dress rings admit $n$-generated ideals that are not $2$-generated.

In the second section we focus on the minimal Dress ring $D$ of $K = \R(X)$. We describe the elements of $D$, prove that it is a Dedekind domain, and characterize its non-principal ideals in terms of their generators. We also show, in Remark \ref{stable}, that $D$ is a simple example of a Dedekind domain of square stable range one, that does not have $1$ in the stable range (cf. \cite{KLW} for these notions).

Finally, in the third section we examine the products of idempotent $2 \times 2$ matrices over $D$. This kind of questions have raised considerable interest, both in the commutative and non-commutative setting. The literature dedicated to these products includes \cite{Fount}, \cite{Ruit}, \cite{SalZan}, \cite{AJLL}, \cite{FL}, \cite{CZ}. We say that $R$ satisfies property \ID2 if every $2 \times 2$ singular matrix over $R$ is a product of idempotent matrices. A natural conjecture, proposed in \cite{SalZan} and investigated in \cite{CZ}, states that, if $R$ satisfies \ID2, then $R$ is a B\'ezout domain (the converse is not true, not even for PID's; see \cite{SalZan}). Note that in \cite{CZ} it is shown that an integral domain $R$ satisfying \ID2 is necessarily Pr\"ufer. 

In this paper we prove, in Theorem \ref{fattorizzazione}, that if $p$, $q$ are elements of $D$ satisfying some conditions on their degrees and roots (recall that $p, q$ are rational functions), then the matrix $\begin{pmatrix}
p & q\\
0 & 0
\end{pmatrix}$ is a product of idempotent matrices.

The question whether $D$ does not satisfy \ID2, in accordance with the conjecture, remains open. We just observe that to verify that a Pr\"ufer domain does not satisfy \ID2 is always a challenging problem; for instance, see \cite{CZZ} and \cite{CZ} to get an idea of the difficulty. As a matter of fact, in those papers, it is shown that matrices as above are not products of idempotent matrices, for suitable choices of $p, q$.

\section{Dress rings}

Let $K$ be a (commutative) field; a subring $D$ of $K$ is said to be a Pr\"ufer-Dress ring (or simply a Dress ring) if $D$ contains every element of the form
\begin{equation} \label{d1} 
{1\over{1 + x^2}},  \qquad  x \in K.
\end{equation}
Of course, the existence of a Dress ring in $K$ implies that the polynomial $X^2 + 1 \in K[ X]$ is irreducible; in particular, if $K$ has characteristic $p \ne 0$, then $p \equiv 3$ modulo $4$. 

A Dress ring $D$ is always a Pr\"ufer domain, since every ideal of $D$ generated by two elements is invertible, a sufficient condition to be Pr\"ufer (see \cite{Gilmer}, Theorem 22.1). Indeed, for any ideal $(a, b)$ in $D$ we have $(a, b)^2 = (a^2 + b^2)D$. This easily derives from the formula, observed by Dress \cite{D},
\begin{equation} \label{d2}
{{2x}\over{1 + x^2}} = {{y^2 - z^2}\over{y^2 + z^2}}  \qquad {\rm if} \ y = x + 1, z = x - 1.
\end{equation} 
In fact, for $x = b/a$ or $x = a/b$, using (\ref{d1}) we get $a^2/(a^2 + b^2) \in D$, $b^2/(a^2 + b^2) \in D$, respectively. From (\ref{d2}) we get $ab/(a^2 + b^2) \in D$ (recall that $1/2 \in D$). It follows that $(a, b)^2 = (a^2, ab, b^2) \subseteq (a^2 + b^2)D$, and the reverse inclusion is trivial.
From the above formula it also follows that $K$ is the field of fractions of $D$.

\medskip

From now on, $K$ will always denote a field not containing square roots of $-1$. For any ring $R$, we denote by $R^*$ its multiplicative subgroup of units.

We are interested in the {\it minimal Dress ring} $D_K$ of $K$, namely
$$
D_K = \Z_{\chi} [1/(1 + x^2) :   x \in K],
$$
where $\chi$ is the characteristic of $K$ and $\Z_\chi = \Z/\chi \Z$ is the prime subring of $K$. (Recall that $\chi \not\equiv 1, 2$ modulo $4$.)

\medskip
We recall now an equivalent definition of minimal Dress rings, in terms of intersections of special valuation domains.
Given a valuation $v$ on the field $K$, we will denote by $V_v$ the corresponding valuation domain and by $\mathfrak{M}_v$ the maximal ideal of $V_v$.

 Dress \cite{D} proved that $D_K$ is the intersection
\begin{equation} \label{inters}
D_K=\bigcap_{v\in B} V_v,
\end{equation}
where $B$ is the family of the valuations over $K$ such that $\sqrt{-1}\notin V_v/\mathfrak{M}_v$. Note that $B \ne \emptyset$, since, by the standing assumption $\sqrt{-1} \notin K$, at least the trivial valuation lies in $B$.

For the sake of completeness, we verify the equality (\ref{inters}). If $v \in B$, then for every $x \in K$ we get $v(1 + x^2) \le 0$, hence $1/(1 + x^2) \in V_v$. It follows that $V_v \supset \Z_{\chi} [1/(1 + x^2) :   x \in K] = D_K$, hence $D_K \subseteq \bigcap_{v\in B} V_v$. Conversely, being a Pr\"ufer domain, $D_K$ is integrally closed, hence it is the intersection of its valuation overrings. Take any valuation overring $V$ of $D_K$, with maximal ideal $\mathfrak M$. Assume, for a contradiction, that $\sqrt{-1}\in V/\mathfrak{M}$. Then there exists $y \in V$ such that $1 + y^2 \in \mathfrak M$, so $1/(1 + y^2) \notin V$. However $1/(1 + y^2) \in D_K \subset V$, impossible. We have got the reverse inclusion $D_K \supseteq\bigcap_{v\in B} V_v$.  

\medskip

Now we examine the basic examples of minimal Dress rings. 

\begin{prop} \label{1}
Let $D = D_K$ be the minimal Dress ring in the field $K$.

(i) If $K$ is an ordered field such that every positive element is a square, then $D = K$.

(ii) If $K = \Q$, then $D = \Z_S$ where $S$ is the multiplicatively closed subset of $\Z$ generated by the primes $p \equiv 1$ modulo $4$. $D$ is an Euclidean domain that is not a valuation domain.
\end{prop}

\begin{proof}

(i) By definition, $D = \Z[1/(1 + x^2) :   x \in \R]$. Take any $\a \in K$ with $0 < \a < 1$. Then $1/\a = 1 + x^2$ for a suitable $x \in \R$, and therefore $\a \in D$. Now take any $\b \in K$, and choose an integer $m$ such that $0 < \b/m < 1$. Then $\b/m \in D$, and so $\b \in D$, as well. We conclude that $D = K$. 

\noindent
(ii) Now $D = \Z[1/(1 + x^2) :   x \in \Q] = \Z[a^2/(a^2 + b^2) :   a, b \in \Z \setminus \{ 0 \}]$. Take any prime number $p$ and assume that $1/p \in D$. Then
$$
1/p = {{m}\over{\prod_i(a_i^2 + b_i^2)}}
$$
for suitable integers $m, a_i, b_i$. It follows that $p$ divides $\prod_i(a_i^2 + b_i^2)$. As well-known, for instance by the properties of the Gaussian integers, this is possible if and only if $p$ is a sum of two squares, if and only if $p \equiv 1$ modulo $4$. Conversely, assume that the prime number $p$ is a sum of two squares, say $p = a^2 + b^2$. Then we get 
$$
1/p = {{1}\over{a^2 + b^2}} =  {{\la a^2}\over{a^2 + b^2}} +  {{\mu b^2}\over{a^2 + b^2}} \in D,
$$
for suitable integers $\la, \mu$, that exist since $a^2, b^2$ are coprime.
We easily conclude that $D = \Z_S$, as in the statement.

In particular, $D$ is Euclidean, since it is a localization of the Euclidean domain $\Z$, and is not a valuation domain, since, for instance, both $3/7$ and $7/3$ do not lie in $D$. 
\end{proof}

The following two propositions apply, respectively, when $K$ is either $\R((X))$ or $\Q((X))$, the fields of Laurent series ($X$ an indeterminate).

\begin{prop} \label{2}
Let $K$ be an Henselian field with respect to the valuation $v$, $V$ the valuation domain of $v$, $\M$ its maximal ideal. If $V/\M$ is an ordered field such that every positive element is a square, then $V$ is the minimal Dress ring in $K$.
\end{prop}

\begin{proof}
Up to isomorphism, we may assume $V/\M \subseteq \R$. Then the field $K$ has characteristic zero and $\Q \subset V$. Let $D_K$ be the minimal Dress ring in $K$. We firstly prove that $V \subseteq D_K$. We start verifying that $\M \subset D_K$. Take any $0\neq x \in \M$. We look for $z \in K$ such that 
$
x = z/(1 + z^2)
$. Consider the polynomial
$\phi = Y^2 -  Y +  x^2 \in V[Y]$, that,
modulo $\M$, has the nonzero simple root $1 + \M$. Since $V$ is Henselian, $\phi$ has a root $y \in V$. 
Setting $z = x/y$, we readily see that $z$ is the element we were looking for. Since $x$ was arbitrary, we conclude that $\M \subset D$. Now take any unit $\eta \in V^*$. Since $\Q \subset V$ and the positive elements of $V/\M$ are squares, we may pick an integer $m$ such that $\eta/m \equiv 1/(1 + z^2)$ modulo $\M$, for some $z \in V$ (see Proposition \ref{1}(i)). We conclude that $\eta \in D$, since $\M \subset D$. 

Let us verify the reverse inclusion. Since every element of $K \setminus V$ is the inverse of an element of $\M$, it suffices to show that $1/x \notin D_K$ for every $x \in \M$. We start noting that, since $V$ is a valuation domain, every element $\a = 1/(1 + z^2)$, $z \in K$ may be written either as $\a =1/(1 + f^2)$, $f \in V$, or $\a = g^2/(1 + g^2)$, $g \in V$. As a consequence, any $r \in D_K$ has the form $r = f/\prod_i(1 + g_i^2)$, for suitable $f, g_i \in V$. We now assume, for a contradiction, that
$
1/x = {f}/{\prod_i(1 + g_i^2)} \in D_K
$, for suitable $f, g_i \in V$. It follows that $\prod_i(1 + g_i^2) \equiv 0$ modulo $\M$, against the hypothesis $V/\M \subseteq \R$. 
\end{proof}

 Following \cite{SalZan}, we say that an integral domain $R$ {\it admits a weak (Euclidean) algorithm}, if for any $a, b \in R$ there exists a (finite) sequence of divisions that starts with $a, b$ and terminates with last remainder zero. Necessarily, such an $R$ is a B\'ezout domain.

\begin{prop} \label{3}
Let $K$ be an Henselian field with respect to the valuation $v$, $V$ the valuation domain of $v$, $\M$ its maximal ideal.
Let $D_K$ be the minimal Dress ring in $K$. If $V/\M \cong \Q$, then $D_K = \Z_S + \M$, where $S$ is the multiplicatively closed subset of $\Z$ generated by the primes $p \equiv 1$ modulo $4$. $D_K$ is neither a valuation domain, nor Noetherian, but admits a weak algorithm. In particular, $D_K$ is a B\'ezout domain. 
\end{prop}

\begin{proof}
Since $V/\M \cong \Q$, the field $K$ has characteristic zero and $\Q \subset V$. Since $V$ is Henselian, in the same way as in Proposition \ref{2} we may show that $\M \subset D_K$. Moreover, like in Proposition \ref{1}(ii) we see that $\Z_S \subset D_K$. It follows that $\Z_S +\M \subseteq D_K$.

Let us verify the reverse inclusion. Take any $r \in D_K$; we may assume that $r \notin \M$. Since $V/\M \cong \Q$, the element $r$ may be written as $r = a + \delta$, where $a \in \Q$ and $\delta \in \M$. Hence, to prove that $D_K \subseteq \Z_S + \M$, it suffices to verify that $a \in \Z_S$. Since $a \in D_K$, we get $a =  f/\prod_i(1 + g_i^2)$, for some $f, g_i \in V$ (cf. the proof of Proposition \ref{2}). Again using $V/\M \cong \Q$, we get $a \equiv  b/\prod_i(1 + c_i^2)$ modulo $\M$, for some $b, c_i \in \Q$, and therefore $a =  b/\prod_i(1 + c_i^2)$, since $\Q \cap \M = 0$. So $a \in \Z_S$, by Proposition \ref{1}(ii). 

Moreover, $D_K$ is not a valuation domain, since $3/7, 7/3 \in \Q \setminus D_K$, and $D_K$ is not Noetherian, since the $D_K$-ideal $\M$ is not finitely generated. To get more examples, we remark that any $D_K$-ideal generated by the set $\{ x/p^n : n > 0 \}$, where $0 \ne x \in \M$ and $p \equiv 3$ modulo $4$, is not finitely generated. Finally, in view of Proposition 6.4 of \cite{SalZan}, $D_K$ admits a weak algorithm, since $\Z_S$ is Euclidean; in particular $D_K$ is a B\'ezout domain.     
\end{proof}

We remark that all the above examples of minimal Dress rings satisfy property ID$_2$, recalled in the introduction (the last one since it admits a weak algorithm, cf. \cite{SalZan}).

\medskip

Let $\mathcal A$ be a set of indeterminates over $\R$, $K = \R(\mathcal A)$ the corresponding field of rational functions.

If $f, g \in \R[\mathcal A]$, we define the degree of $f/g$ in the natural way, namely: $\deg(f/g) = \deg(f) - \deg(g)$. We remark the following useful, readily verified property: 
if $f_1, g_1, f_2, g_2 \in \R[\mathcal A]$ are sums of squares, then $\deg(f_1/g_1 + f_2/g_2) = \sup \{ \deg(f_1/g_1) , \deg(f_2/g_2) \}$.

Recall that a field $F$ is {\it formally real} if $1 + \sum x_i^2 \ne 0$ for every choice of $x_i \in F$ (in particular $\sqrt{-1} \notin F$). A valuation $v$ on the field $K$ is said to be formally real if the residue field $V_v/\M_v$ is formally real. The existence of formally real valuations on the field $K$ implies that $K$ is formally real; in particular, $\sqrt{-1} \notin K$. 
In the 1979 paper \cite{Schulting} Sch\"ulting considered the integral domain
$$
R_K = \bigcap_{v \in C} V_v,
$$ 
where $C$ is the set of the formally real valuations on $K$. From the equality (\ref{inters}) it immediately follows that $R_K$ contains the minimal Dress ring $D_K$. In fact, if $v$ is a formally real valuation, then $\sqrt{-1} \notin V_v/\M_v$, hence $V_v \supset D_K$. So the integral domain $R_K$ is a Dress ring in our sense, hence it is a Pr\"ufer domain; it satisfies several interesting properties (see Chapter II of \cite{FHP} for a detailed description). 

The main Sch\"ulting's purpose was to exhibit examples of Pr\"ufer domains that admit finitely generated ideals that are not $2$-generated. Indeed, if $K = \R(X_1, \dots, X_n)$, it was proved that the fractional ideal $(1, X_1, \dots, X_n)$ of $R_K$ cannot be generated by less then $n+1$ elements (see \cite{Schulting} and \cite{OR}). This result is deep and difficult to prove: techniques of algebraic geometry are required. 

 If $K = \R[X_1, \dots, X_n]$ with $n \ge 2$ we conclude that also $D_K$ admits finitely generated ideals ideals that are not $2$-generated. For instance, if $n = 2$ the fractional $D_K$-ideal $(1, X_1, X_2)$ is not $2$-generated over $D_K$, otherwise the fractional $R_K$-ideal would be $2$-generated over $R_K \supset D_K$, against Sch\"ulting results in \cite{Schulting}.

\begin{prop} \label{Sch}
In the above notation, let $K = \R(X_1, \dots, X_n)$.

(i) If $n = 1$, then $R_K$ coincides with $D_K$;

(ii) if $n \ge 2$, then $R_K$ properly contains $D_K$.
\end{prop}

\begin{proof}
(i) Let $K = \R(X)$. By (\ref{inters}) $R_K = D_K$ if and only if for every valuation overring $V$ of $D_K$, the residue field $V/\M$ is formally real. Assume, for a contradiction, that there exist $r_1, \dots, r_m \in V \setminus \M $ such that 
$1 + \sum_{i = 1}^m r_i^2 \in \M$.
We may write $r_i = f_i/g$, where $f_i, g \in \R[X]$. Let $v$ be the valuation determined by $V$; recall that $\sqrt{-1} \notin V/\M$. We firstly consider the case where $v(X) \ge 0$, hence $v(g) \ge 0$, and the above equality yields 
$\f = g^2 +  \sum_{i = 1}^m f_i^2 \in \M$.
Factorizing $\f$ in $\R[X]$, we may write $
\f = \a \prod_j (1 + s_j^2) \in \M$,
for suitable $\a > 0$ and linear polynomials $s_j$. We reached a contradiction, since $\sqrt{-1} \notin V/\M$ implies that  $1 + s_j^2 \notin \M$, for all $j$.

It remains to examine the special case where $v(X) < 0$. Under the present circumstances $v = - \deg$, hence $f_i/g \in V \setminus \M$ implies $\deg \, f_i = \deg \, g$ and $  \deg \, \f = 2 \, \deg \, g$. Therefore $v(\f/g^2) = - \deg(\f/g^2) = 0$, so $1 + \sum_{i = 1}^m r_i^2 \notin \M$, another contradiction.

(ii) Let $K = \R(X_1, \dots, X_n)$ with $n \ge 2$. The Sch\"ulting ring $R_K$ properly contains $D_K$ if there exists a valuation overring $V$ of $D_K$ such that $V/\M$ is not formally real. Let us consider the irreducible polynomial $
f = 1 + X_1^2 + \dots + X_n^2 \in \R[X_1, \dots, X_n]$.
 Let $w$ be the rank-one discrete valuation determined by $f$, i.e., the valuation that extends the assignments $w(f) = 1$ and $w(g) = 0$ if $f$ does not divide $g \in \R[X_1, \dots, X_n]$. By the definition, $w$ is not a formally real valuation, hence the valuation domain $V$ determined by $w$ does not contain  $R_K$. To get our conclusion, it suffices to prove that $V$ is a valuation overring of $D_K$. Assume, for a contradiction, that $\sqrt{-1} \in V/\M$, say $a^2/b^2 + 1 \in \M$, with $a, b \in \R[X_1, \dots, X_n]$, $b$ coprime with $f$. Then there exist $c, d \in K[X_1, \dots, X_n]$, $d$ coprime with $f$, such that $d(a^2 + b^2) = f c b^2$.
By coprimality we get $a^2 + b^2 = f e$, for some $e \in \R[X_1, \dots, X_n]$. Factorizing the last equality in $\mathbb C[X_1, \dots, X_n]$, we derive that $f$ divides $a + ib$ (and $a - i b$, as well) in $\mathbb C[X_1, \dots, X_n]$, since $f$ is irreducible in $\mathbb C[X_1, \dots, X_n]$. Say $a + i b = f(g_1 + i g_2)$, with $g_1, g_2 \in \R[X_1, \dots, X_n]$. It follows that $b = f g_2$, hence $b$ is not coprime with $f$, impossible.
\end{proof}

\section{The minimal Dress ring of $\R(X)$}

In this section we focus on the minimal Dress rings $D$ of the field of rational functions $\R(X)$. Our aim is to give a complete description of the elements of this ring. We will also prove that $D$ is a Dedekind domain, i.e. a Noetherian Pr\"ufer domain, and we characterize the non-principal ideals of $D$.

Let $\G = \{\alpha\prod_i \g_i \}$, where the $\g_i$ are monic degree-two polynomials irreducible over $\R[X]$ and $0 \ne \alpha$ is a real number. Of course, $\G$ coincides with the set of the polynomials in $\R[X]$ that have no roots in $\R$. In what follows we also use the notation $\G^+ = \{f \in \R[X] : f(r) > 0 , \forall \, r \in \R \}$, and, correspondingly, $\G^- =  \{- f  : f \in \G^+ \}$.

\medskip

It is easy to characterize the elements of $D$. 

\begin{prop} \label{charac}
Let $D$ be the minimal Dress ring of $\R[X]$. Then
$$
D = \{f/\g : f \in \R[X], \g\in \G , \deg \, f \le \deg \, \g \}.
$$
\end{prop} 

\begin{proof}
Since $\R \subset \R(X)$, we get $\R \subset D$, and therefore $D = \R[a^2/(a^2 + b^2)]$, where $a, b \in \R[X]$. Say $f/\g \in D$, for coprime polynomials $f, \g \in \R[X]$. It is then clear that $\g \in \G$, since in $\R[X]$ a sum of two squares lies in $\G$. Since the generators $a^2/(a^2 + b^2)$ have degree $\le 0$, it readily follows that $\deg(f/\g) \le 0$. Hence $D \subseteq \{f/\g : \g\in \G , \deg \, f \le \deg \, \g \}$. To verify the reverse inclusion, it suffices to show that, for any assigned $\g \in \G$ of degree $2n$, say, the rational function $X^m/\g$ lies in $D$, for $0 \le m \le 2n$. We can write $ \g = \a \prod_{i = 1}^n(1 + (r_iX - s_i)^2)$, for suitable real numbers $\a \ne 0,  r_i \ne 0, s_i$. Then $1/(1 + (r_iX - s_i)^2)$,  $(r_i X - s_i)/(1 + (r_iX - s_i)^2)$ and $(r_i X - s_i)^2/(1 + (r_iX - s_i)^2)$ lie in $D$, hence easy computations show that $X/(1 + (r_iX - s_i)^2) \in D$, $X^2/(1 + (r_iX - s_i)^2) \in D$.  We readily conclude that $X^m/\g \in D$. 
\end{proof}

From the preceding characterization, we immediately see that $u \in D^*$ if and only if $u = \g_1/\g_2$, where $\g_1, \g_2 \in \G$ and $\deg \, \g_1 = \deg \, \g_2$. Moreover, any $r \in D$ is associated to an element of the form $f/\g$, where $f \in \R[X]$ is a product of linear factors.

\begin{prop} \label{square}
Let $D$ be the minimal Dress ring of $\R(X)$. Let $J = (r_i : i \in \Lambda)$ be an ideal of $D$. Then $J^2 = (r_i^2 : i \in \Lambda)$. 
\end{prop}

\begin{proof}
It is clear that $J^2 \supseteq (r_i^2 : i \in \Lambda)$. Conversely, take a typical generator $r_j r_k$ of $J^2$. We know that $r_j r_k \in (r_j, r_k)^2 =  (r_j^2 + r_k^2)D \subseteq (r_i^2 : i \in \Lambda)$. We conclude that $(r_i^2 : i \in \Lambda) \supseteq J^2$.
\end{proof}

\begin{thm} \label{Dedekind}
The minimal Dress ring $D$ of $\R(X)$ is a Dedekind domain.
\end{thm}

\begin{proof}
Since $D$ is a Pr\"ufer domain, in order to verify that it is actually a Dedekind domain, it suffices to show that $D$ is Noetherian.

Let $J \ne 0$ be an ideal of $D$. We will prove that $J^2$ is a principal ideal, so $J$ is invertible, and therefore finitely generated. 

Say $J = (r_i = f_i/\g_i : i < \Lambda)$, where $f_i \in \R[X]$, $\g_i \in \G$, $\deg(r_i) \le 0$. Possibly replacing $J$ with an isomorphic ideal, we may safely assume that the $f_i$ are coprime, i.e., there exists a finite subset $A$ of $\La$ such that $1 = {\rm gcd}(f_j : j \in A)$. By Proposition \ref{square} we know that $J^2 = (r_i^2: i \in \La)$. Let $s = \sum_{j \in A} r_j^2$. Since $\deg(r_i) \le 0$ for every $i \in \La$, we may choose $s$ of maximal degree. Let us verify that $J^2 = s D$. It suffices to show that $r_m^2 = (f_m/\g_m)^2 \in s D$ for every generator $r_m^2$ of $J^2$. Say $s = g/\g^2$, where $\g = \prod_{j \in A}\g_j$ and $g = \sum_{j \in A} f_j^2 \g^2/\g_j^2$. Note that $g$ is a sum of squares and has no roots in $\R$, since the $f_j$ have no common root in $\R$. 
We get
$$ 
(f_m/\g_m)^2 / s = {{ f_m^2 \g^2}\over{g \g_m^2}} \in D.
$$
In fact $g \g_m^2 \in \G$, since $g$ has no roots in $\R$. Moreover $\deg(f_m^2/\g_m^2) \le \deg s$, otherwise the choice $s_1 = s + (f_m/\g_m)^2$ would contradict the assumption that $s$ has maximal degree. (Recall that $\deg(s + r_m^2) = \sup \{ \deg(s) , \deg(r_m^2) \}$, since $s$ is a sum of squares.)
\end{proof}

Since $D$ is a Dedekind domain, all its ideals are generated by two elements. The next proposition gives a characterization of the ideals over $D$ that are exactly $2$-generated.

\begin{prop} \label{4}
Let $f/\g$ and $g/\g$ be elements of $D$, with $f, g \in \R[X]$, $\g \in \G$. Let $M=\gcd(f,g)$, say $f=M f'$, $g=M g'$. Then the ideal $\left(f/\g,g/\g\right)$ is not principal if and only if $s = \sup \{ \deg \,f' , \deg \,g' \}$ is odd.
\end{prop}

\begin{proof}
It suffices to prove that the fractional ideal $\left(f',g'\right)$ is exactly $2$-generated as a $D$-module if and only if $s = \sup \{ \deg \,f' , \deg \,g' \}$ is odd. In any case, $f'^2+g'^2 \in \G$ and $\deg(f'^2+g'^2)=2s$, since $f', g'$ are coprime.

We firstly consider the case when $\deg(f'^2+g'^2)=2s$ with $s$ odd. Assume by contradiction that $\left(f',g'\right)=hD$ for some $h\in\R(X)$. It follows that $\left(f',g'\right)^2=(f'^2+g'^2)D=h^2 D$, hence $f'^2+g'^2=uh^2$, where $u \in D^*$. It follows that the rational function $h$ has neither zeros nor poles in $\R$, hence $h = \g/\delta$, with $\g, \delta \in \G$. Therefore $h$ has even degree, say $\deg \, h = 2n$. Then $u= (f'^2+g'^2)/h^2$ shows that $\deg \, u = 2s - 4n \ne 0$. But this is impossible, since $u$ is a unit of $D$, hence $\deg \, u = 0$.

Conversely, let us assume that $s$ is even. If $s=0$, both $f'$ and $g'$ lie in $\R$ and $\left(f',g'\right)=D$. If $s > 0$, take $h \in \G$ of degree $s$, possible since $s$ is even. Then $f'/h$ and $g'/h$ are both elements of $D$, so $f', g' \in hD$, whence $(f', g') \subseteq hD$.

Moreover, let
$$
u=\frac{f'^2+g'^2}{h^2}.
$$
Then the choice of the degrees shows that $\deg \, u = 0$, so $u \in D^*$. From the preceding equality we get
\begin{equation*}
h=\frac{(f'^2+g'^2)}{u h}
= u^{-1}\frac{f'}{h} f'+u^{-1} \frac{g'}{h} g',
\end{equation*}
where $u^{-1} {f'}/{h}, u^{-1}{g'}/{h}\in D$, whence $h \in \left(f',g'\right)$. We have got the reverse inclusion $h D \subseteq (f', g')$.
\end{proof}

\begin{rem} 
One may consider the minimal Dress rings $D_K$ of $K = \R(\mathcal A)$, where $\mathcal A$ is any set of indeterminates. However, the property of being Noetherian is specific of the Dress ring of $\R(X)$. In fact, by Sch\"ulting's result in \cite{Schulting}, generalized by Olberding and Roitman in \cite{OR}, if $\mathcal A$ contains more than one indeterminate, then $D_K$ contains finitely generated ideals that are not $2$-generated. Therefore $D_K$ cannot be Noetherian, otherwise, being a Pr\"ufer domain, it should be Dedekind, hence all its ideals should be generated by two elements. 
\end{rem}

\begin{rem} \label{stable} 
It is worth noting that the minimal Dress ring $D$ of $\R(X)$ is a simple example of a Dedekind domain of square stable range one, but without $1$ in the stable range. We refer to \cite{KLW} for a thorough investigation on these notions. The ring $D$ appears to be easier than their examples, based on Theorem 5 of Swan's paper \cite{Swan_n_gen}.

Take $a, b \in D$ such that $(a, b) = D$. Then, from $D = (a, b)^2 = (a^2 + b^2)D$, it follows that $a^2 + b^2 \in D^*$, and therefore $D$ has square stable range one.

Let us show that $D$ does not have $1$ in the stable range. Take the following elements of $D$: $a = X/\g$, $b = (X^2 - 1)/\g$, where $\g = 1 + X^2$. Then $D = (a, b)$, since $a^2 + b^2$ is a unit of $D$. Let us show that $a + bz \notin D^*$,
for every $z \in D$. Let us assume, for a contradiction, that $a + bz \in D^*$, where $z = f/\delta$, for suitable $f \in \R[X]$ and $\delta \in \G^+$. Being a unit, $a + bz$ has no roots in $\R$. Equivalently, $f_1 = X \delta + (X^2 - 1) f$ has no roots. However, $f_1(1) = \delta(1) > 0$, and $f_1(-1) = - \delta(-1) < 0$, so $f_1$ does have roots. We reached a contradiction.  
\end{rem}

\section{Matrices over $D$ that are products of idempotent matrices}

Proposition \ref{4} and Theorem \ref{Dedekind} show that $D$, the minimal Dress ring of $\R(X)$, is a Dedekind domain which is not a principal ideal domain. Hence we expect that $D$ does not satisfy property {\ID2}, in support of the conjecture recalled in the introduction. This difficult question is left open. Actually, in this section we will focus on properties of the entries of $2 \times 2$ singular matrices over $D$, that guarantee factorization into products of idempotent.

We start with some considerations that hold over an arbitrary integral domain $R$. 

It is easy to prove that a singular nonzero matrix $\begin{pmatrix}
a & b\\
c & d
\end{pmatrix}$ is idempotent if and only if $d = 1 - a$ (cf. \cite{SalZan}). 

Now we remark some useful factorizations into idempotents.

\begin{equation} \label{0pq}
\begin{pmatrix}
0 & q\\
0 & 0
\end{pmatrix}=\begin{pmatrix}
1 & 0\\
0 & 0
\end{pmatrix}\begin{pmatrix}
0 & q\\
0 & 1
\end{pmatrix} \ ; \ \begin{pmatrix}
p & 0\\
0 & 0
\end{pmatrix}=\begin{pmatrix}
1 & -1\\
0 & 0
\end{pmatrix}\begin{pmatrix}
1 & 0\\
1-p & 0
\end{pmatrix}
\end{equation}

\begin{equation} \label{p divides q}
\begin{pmatrix}
p & rp\\
0 & 0
\end{pmatrix}=\begin{pmatrix}
p & 0\\
0 & 0
\end{pmatrix}\begin{pmatrix}
1 & r\\
0 & 0
\end{pmatrix} = \begin{pmatrix}
1 & -1\\
0 & 0
\end{pmatrix}\begin{pmatrix}
1 & 0\\
1-p & 0
\end{pmatrix} \begin{pmatrix}
1 & r\\
0 & 0
\end{pmatrix}
\end{equation}

A pair of elements $p, q$ of $R$ is said to be an {\it idempotent pair} if $(p \ q)$ is the first row of an idempotent matrix. In this case we get the factorization
\begin{equation}\label{pair}
\begin{pmatrix}
p & q\\
0 & 0
\end{pmatrix}=\begin{pmatrix}
1 & 0\\
0 & 0
\end{pmatrix}\begin{pmatrix}
p & q\\
r & 1-p
\end{pmatrix},
\end{equation}
where $rq = p(1 - p)$.

We note that, obviously, every matrix similar to an idempotent matrix is idempotent, hence a singular matrix $\mathbf{S}$ is a product of idempotent matrices if and only if any matrix similar to $\mathbf{S}$ is a product of idempotents.

\begin{lem} \label{scambio}
Let $R$ be a domain, $p, q \in R$. The matrix 
$\mathbf A = \begin{pmatrix}
p & q\\
0 & 0
\end{pmatrix}$
is a product of idempotent matrices if and only if such is
$\mathbf{B} = \begin{pmatrix}
q & p\\
0 & 0
\end{pmatrix}.$ 
\end{lem}

\begin{proof}
Assume that $\mathbf A$ is a product of idempotents. We get 
$$
 \begin{pmatrix}
0 & 1\\
1 & 0
\end{pmatrix}\mathbf A \begin{pmatrix}
0 & 1\\
1 & 0
\end{pmatrix} = \begin{pmatrix}
0 & 0\\
q & p
\end{pmatrix} = \mathbf S.
$$
Hence $\mathbf S$ is also a product of idempotents. However 
$$
\mathbf{B} = \begin{pmatrix}
1 & 1\\
0 & 0
\end{pmatrix} \mathbf S,
$$
hence $\mathbf{B}$ is a product of idempotents. The argument is reversible.
\end{proof}

Now we get back to $D$, the minimal Dress ring of $\R(X)$. 
We will provide explicit factorizations into idempotent matrices of particular classes of singular matrices of the form $\begin{pmatrix}
p & q\\
0 & 0
\end{pmatrix}$, where the entries $p, q \in D$ satisfy rather natural mutual relations in terms of their roots and degrees. 

The next lemma is crucial.

\begin{lem}\label{grado uguale}
Let $x,y$ be two non-zero polynomials in $\R[X]$ with $\deg\,(x)=\deg\,(y)$. 

\begin{enumerate}[(a)]

\item If $y(u)>0$ (or $y(u)<0$) for every $u$ root of $x$, then there exists $\beta \in \G$ such that $\delta=x^2+ y \beta\in \G^+$, $\deg\, x - 1 \le \deg \b \le \deg \, x = \deg\,\delta/2$.

\item If $x(z)>0$ (or $x(z)<0$) for every $z$ root of $y$, then there exists $\eta \in \G$ such  that $\delta= x\eta+y^2 \in \G^+$ and $\deg \, y - 1 \le \deg\,\eta \le \deg\,(y)=\deg \delta/2$.
\end{enumerate}
\end{lem}

\begin{proof}
We will prove (a), since (b) is analogous, exchanging the roles of $x$ and $y$. We distinguish the cases $x\in\G$ and $x\notin \G$.

\medskip
FIRST CASE: $x\in\G$.

Since $x$ has no roots in $\R$, then $x^2$ is always positive. Now take $ \beta' \in \G$ of degree equal to either $\deg \, x$ or $\deg \, x - 1$, in accordance with the parity of the degree of $x$, and such that the leading coefficient of $x^2 - \b' y$ is positive. Then 
$\lim_{X \to \pm \infty}(x^2 - \b' y) = + \infty$, hence there exists $M > 0$ such that $x^2 - \b'y > 0$ outside the interval $I = [-M, M]$. Let $k > 0$ be the minimum of $x^2$ and $h > 0$ the maximum of $|\b' y|$ in the interval $I$, and pick an integer $m \ge 2$ such that $k/mh < 1$. Then, inside the interval $I$, we get $|\b' y|/mh < 1$, and
$$
x^2 \ge k > k |\b' y|/mh.
$$
Therefore, $x^2 - k \b' y/mh > 0$ inside the interval $I$. Moreover, outside the interval $I$, we have $x^2 - k \b'y/mh > 0$, as well. This is trivial for the points where $k \b'y/mh < 0$, and, for the points where $k \b'y/mh > 0$, we get $x^2 - k\b'y/mh > x^2 - \b'y > 0$,
since $k/mh < 1$. So the polynomial $\b = - \b' k/mh$ satisfies our requirements.

\medskip
SECOND CASE: $x\notin\G$.

Let $u_1< u_2 < \dots < u_n$ be the distinct roots of $x$. Now we pick $\b'  \in \G$ of degree equal to either $\deg \, x$ or $\deg \, x - 1$, in accordance with the parity of the degree of $x$, and such that the leading coefficient of $x^2 - \b' y$ is positive, and $\b' y(u_i) < 0$ for $1 \le i \le n$. As in the First Case, we may take an interval $I = [- M, M]$ such that $x^2 - \b' y > 0$ outside $I$; under the present circumstances, we also choose $M > 0$ such that $- M < u_1$, $u_n < M$. For $1 \le i \le n$, since $\b' y(u_i) < 0$, we may choose disjoint open intervals $I_i$ containing $u_i$, such that $\b' y < 0$ in $I_i$. We also require that $-M <  \inf  I_1$ and $\sup  I_n < M$.  For $1 \le i < n$, let $J_i$ be the closed interval whose end points are $\sup \, I_i$ and $\inf  I_{i+1}$; we also define $J_0 = [-M, \inf  I_1]$, and $J_n = [\sup I_n, M]$. 

Now, for $0 \le i \le n$, let $k_i$ be the minimum of $x^2$ and $h_i$ the maximum of $|\b' y|$ in $J_i$. So, arguing as in the First Case, we get
$$
x^2 - k_i \b' y/mh_i > 0
$$
inside the interval $J_i$, where $m > 0$ is an integer enough large to work for all $i$. We conclude that there exists a real number $s \in ]0, 1[$ such that $x^2 - s \b' y > 0$ in $\bigcup_i J_i$. However note that $x^2 - s \b' y > 0$ in $\bigcup_i I_i$, since $\b' y < 0$ in each $I_i$. We conclude that $x^2 - s \b' y > 0$ in $I = [-M, M]$, and, arguing as in the First Case, it is positive also outside $I$. We conclude that the polynomial $\b = - s \b'$ satisfies our requirements.  
\end{proof}

\begin{thm}\label{fattorizzazione}
Let $p$ and $q$ be two elements of $D$. Then the matrix $\begin{pmatrix}
p & q\\
0 & 0
\end{pmatrix}$ is a product of idempotent matrices if one of the following holds:
\begin{enumerate}[(i)]
\item $\deg\,p \ge \deg\,q$ and $q(u)>0$ (or $q(u)< 0$) for every $u$ root of $p$

\item $\deg\,q \ge \deg\,p$ and $p(z)>0$ (or $p(z)< 0$) for every $z$ root of $q$.
\end{enumerate}
\end{thm}

\begin{proof}
Recall that, by Lemma \ref{scambio}, $\begin{pmatrix} 
p & q\\
0 & 0
\end{pmatrix}$ is a product of idempotents if and only if such is $\begin{pmatrix}
q & p\\
0 & 0
\end{pmatrix}$. Hence we may safely assume that (i) holds. Moreover, note that
\begin{equation} \label{similar}
\begin{pmatrix}
1 & - 1\\
0 & 1
\end{pmatrix}
\begin{pmatrix}
p & q\\
0 & 0
\end{pmatrix}
\begin{pmatrix}
1 & 1\\
0 & 1
\end{pmatrix} =
\begin{pmatrix}
p & p + q\\
0 & 0
\end{pmatrix}.
\end{equation}
Hence, in case $\deg(p) > \deg(q)$, it suffices to prove that $\begin{pmatrix}
p & p + q\\
0 & 0
\end{pmatrix}$, is a product of idempotents. Note also that $p + q$ keeps the same sign on every root $u$ of $p$, so (i) holds replacing $q$ with $p + q$. In conclusion, we may assume that $\deg \, p = \deg \, q$. Say $p = x/\g$, $q = y/\g$, with $x, y \in \R [X]$, $\g \in \G$. 

As a further reduction, we may assume that $\deg \, \g \le \deg \,x + 1$. Otherwise, take $\tau \in \G$ such that $\deg \, x \le \deg \, \tau \le \deg \, x + 1$. Then $\tau/\g \in D$ and
$$
\begin{pmatrix}
x/\g & y/\g\\
0 & 0
\end{pmatrix} =
\begin{pmatrix}
\tau/\g & 0\\
0 & 0
\end{pmatrix}
\begin{pmatrix}
x/\tau & y/\tau\\
0 & 0
\end{pmatrix}.
$$
Hence, by (\ref{0pq}),  the first member of the above equality is a product of idempotents if such is the second factor of the second member.

After the preceding reductions, we have got in the position to apply Lemma \ref{grado uguale} to $x, y$. We find $\b \in \G$ such that
$
\delta = x^2 + y \b \in \G^+
$
where $\deg \, \b = \deg \, x$ when $\deg \, x$ is even, and $\deg \, \b = \deg \, x - 1$ when $\deg \, x$ is odd. 

Then, if $\deg \, x $ is even, we get 
$
\deg \, \delta = \deg(x^2) =  \deg \, \g + \deg \, \b,   
$
while, if $\deg \, x$ is odd, we get
$
\deg \, \delta = \deg(x^2) = \deg \, x + 1 + \deg \, \b = \deg \, \g + \deg \, \b,
$
hence, in both cases, $\delta/\g \b \in D^*$ and $x^2/\delta, yx/\delta, x\beta/\delta, y\beta/\delta \in D$.

Since $1 - x^2/\g = y \b/\g$, $\mathbf{T}=\begin{pmatrix}
x^2/\delta & yx/\delta\\
x\beta/\delta & y\beta/\delta
\end{pmatrix}$ is an idempotent matrix over $D$. Hence, by (\ref{0pq}) and the factorization 
\[\begin{pmatrix}
x/\g & y/\g \\
0 & 0
\end{pmatrix}=\begin{pmatrix}
\delta/\g\eta & 0 \\
0 & 0
\end{pmatrix}\mathbf{T},
\]we conclude that $\begin{pmatrix}
p & q\\
0 & 0
\end{pmatrix}$ is a product of idempotent matrices over $D$.
\end{proof}

\begin{cor}
In the above notation, let $p=x/\g$ and $q=y/\g$ be two elements of $D$. If both $\deg\,x$ and $\deg\,y$ are $\le 1$, then $\begin{pmatrix}
p & q\\
0 & 0
\end{pmatrix}$ is a product of idempotent matrices.
\end{cor}
\begin{proof}
We may assume $x \ne 0 \ne y$ are nonzero. If $x$ and $y$ have no common roots, then our statement follows from Theorem \ref{fattorizzazione}, since either (i) or (ii) trivially holds when the degrees are $\le 1$. Otherwise, we get $y=rx$, for some $r\in \R$, hence (\ref{p divides q}) yields the desired conclusion.
\end{proof}

It is worth showing that, in case of ``small degrees of the numerators'', we get factorizations into idempotents even when $p$, $q$ have common roots. 

\begin{prop}
In the above notation, let $p = x/\g, q = y/\g \in D$ with $\deg\,x=\deg\,y=2$. If $M=\gcd(x,y)\notin\R$, then $\mathbf{A} =\begin{pmatrix}
p & q\\
0 & 0
\end{pmatrix}$ is a product of idempotent matrices.
\end{prop}

\begin{proof}
If $\deg\,M=2$, the factorization follows from (\ref{p divides q}).

If $\deg\,M=1$, we may assume, up to a linear change of coordinates, that $M=X$. Take $r \in \R$ such that $x + r y = sX$, where $0 \ne s \in \R$. Take $\d \in \G$ such that $\d - x$ has degree $1$, and consider the polynomial $z = s^{-1}(\d - x) x/X \in \R[X]$. Then the matrix $ \begin{pmatrix}
x/\d & sX/\d\\
z/\d & 1 - x/\d
\end{pmatrix}$ is idempotent, hence $\mathbf{B} = \begin{pmatrix}
x/\d & sX/\d\\
0 & 0
\end{pmatrix}$ is a product of idempotents. However, analogously to (\ref{similar}), we may verify that $\mathbf{A}$ is similar to $\mathbf{B} =  \begin{pmatrix}
x/\d & y/\d\\
0 & 0
\end{pmatrix} $. Finally, since $\deg \, \d = 2$, the equality
$
\mathbf{A} =  \begin{pmatrix}
\d/\g & 0\\
0 & 0
\end{pmatrix} \mathbf{B} 
$
shows that $\mathbf{A}$ is a product of idempotents (cf., the proof of Theorem \ref{fattorizzazione}).
\end{proof}


\bibliographystyle{plain}

\end{document}